\documentclass[reqno]{amsart}
\usepackage{amsmath, amssymb, amsfonts}
\newtheorem{theorem}{Theorem}
\newtheorem{corollary}{Corollary}
\newtheorem{lemma}{Lemma}

\theoremstyle{remark}
\newtheorem{remark}{Remark}
\theoremstyle{definition}
\newtheorem{define}{Definition}
\newtheorem{example}{Example}

\begin{document}

\title[On the automorphism group of a certain infinite type domain]{On the automorphism group of a certain infinite type domain in $\mathbb C^2$}

\author{Ninh Van Thu}

\thanks{ The research of the author was supported in part by a  grant  of Vietnam National University at Hanoi, Vietnam.}

\address{ Department of Mathematics, Vietnam National University at Hanoi, 334 Nguyen Trai, Thanh Xuan, Hanoi, Vietnam}
\email{thunv@vnu.edu.vn}

\subjclass[2010]{Primary 32M05; Secondary 32H02, 32H50, 32T25.}
\keywords{Holomorphic vector field, real hypersurface, infinite type point.}
\begin{abstract}
In this article, we consider an infinite type domain $\Omega_P$ in $\mathbb C^2$. The purpose of this paper is to investigate the holomorphic vector fields tangent to an infinite type model in $\mathbb C^2$ vanishing at an infinite type point and to give an explicit description of the automorphism group of $\Omega_P$.
\end{abstract}
\maketitle

\section{Introduction}

Let $D$ be a domain in $\mathbb C^n$. An automorphism of $D$ is a biholomorphic self-map. The set of all automorphisms of $D$ makes a group under composition. We denote the automorphism group by $\mathrm{Aut}(D)$. The topology on $\mathrm{Aut}(D)$ is that of uniform convergence on compact sets (i.e., the compact-open topology).

It is a standard and classical result of H. Cartan that if $D$ is a bounded domain in $\mathbb C^n$ and the automorphism group of $D$ is noncompact then there exist a point $x \in D$, a point $p \in \partial D$, and automorphisms $\varphi_j \in \mathrm{Aut}(D)$ such that $\varphi_j(x) \to p$. In this circumstance we call $p$ a {\it boundary orbit accumulation point}. 

In $1993$, Greene and Krantz \cite{GK} posed a conjecture that for a smoothly
bounded pseudoconvex domain admitting a non-compact automorphism group, the point orbits can
accumulate only at a point of finite type in the sense of Kohn, Catlin, and D'Angelo (see \cite{D, Kr01} for this concept). For this conjecture, we refer the reader to \cite{IK}.

One of the evidence for the correctness of Greene-Krantz's conjecture is provided in \cite{Ka94}. H. Kang \cite{Ka94} proved that the automorphism group $\mathrm{Aut}(E_P)$ is compact, where $E_P$ is a special kind of  Hartogs  domains
$$
E_P=\{(z_1,z_2)\in \mathbb C^2\colon  |z_1|^2+P(z_2)<1\}\Subset \mathbb C^2,
$$
where $P$ is a real-valued, $\mathcal{C}^\infty$-smooth, subharmonic function satisfying:
\begin{itemize}
\item[(i)] $P(z_2)>0$ if $z_2\ne 0$,
\item[(ii)] $P$ vanishes to infinite order only at the origin.             
\end{itemize}                                                   
Note that $E_P$ is of infinite type along the points $(e^{i\theta},0)\in bE_P$ and $(e^{i\theta},0)$ are the only points of infinite type. 

Recently, S. Krantz \cite{Kr12} showed that the domain
$$
\Omega:=\{z\in \mathbb C^n\colon |z_1|^{2m_1}+ |z_2|^{2m_2}+\cdots+ |z_{n-1}|^{2m_{n-1}}+\psi(|z_n|)<1\},
$$
where the $m_j$ are positive integers and where $\psi$ is a real-valued, even, smooth, monotone-and-convex-on-$[0,+\infty)$ function of a real variable with $\psi(0)=0$ that vanishes to infinite order at $0$, has compact automorphism group. In fact, the only automorphisms of $\Omega$ are the rotations in each variable separately (cf. \cite{GK, IK06}).

We would like to emphasize here that the automorphism group of a domain in $\mathbb C^n$ is not easy to describe explicitly; besides, it is unknown in most cases. In this paper, we are going to compute the automorphism group of  an infinite type model  
$$
\Omega_P:=\{(z_1,z_2)\in \mathbb C^2\colon \rho(z_1,z_2)= \mathrm{Re}~z_1+P(z_2)<0\},
$$
where  $P:\mathbb C\to \mathbb R$ is a $\mathcal{C}^\infty$-smooth function satisfying: 
\begin{itemize}
\item[(i)] $P(z)=q(|z|)$ for all $z\in \mathbb C$, where $q:[0,+\infty)\to \mathbb R$ is a function with $q(0)=0$ such that it is strictly increasing and convex on $[0,\epsilon_0)$ for some $\epsilon_0>0$, and
\item[(ii)] $P$ vanishes to infinite order at $0$.
\end{itemize}
It is easy to see that $(it,0), t\in \mathbb R$, are points of infinite type in $b \Omega_P$, and hence $\Omega_P$ is of infinite type.

In order to state the first main result, we recall the following terminology. A {\it holomorphic vector field} in $\mathbb C^n$ takes the form
$$
H = \sum_{k=1}^n h_k (z) \frac{\partial}{\partial z_k}
$$
for some functions $h_1, \ldots, h_n$ holomorphic in $z=(z_1, \ldots, z_n)$.  A smooth real
hypersurface germ $M$ (of real codimension 1) at $p$ in $\mathbb C^n$ takes a defining
function, say $\rho$, such that $M$ is represented by the equation $\rho(z)=0$.  The holomorphic vector field $H$ is said to be {\it tangent} to $M$ if its real part $\hbox{Re }H$ is tangent to $M$, i.e., $H$ satisfies the equation 
\begin{equation}\label{maineq} 
(\hbox{Re }H)\rho(z) = 0~\text{for all}~z\in M.
\end{equation}

The first aim of this paper is to prove the following theorem, which is a characterization of tangential holomorphic vector fields.
\begin{theorem}\label{Th1} Let $P:\mathbb C\to \mathbb R$ be a $\mathcal{C}^\infty$-smooth function satisfying 
\begin{itemize}
\item[(i)] $P(z)=q(|z|)$ for all $z\in \mathbb C$, where $q:[0,+\infty)\to \mathbb R$ is a function with $q(0)=0$ such that it is strictly increasing and convex on $[0,\epsilon_0)$ for some $\epsilon_0>0$, and
\item[(ii)] $P$ vanishes to infinite order at $0$.
\end{itemize}
If $H=h_1(z_1,z_2)\frac{\partial}{\partial z_1}+h_2(z_1,z_2)\frac{\partial}{\partial z_2}$ with $H(0,0)=0$ is holomorphic in $\Omega_P\cap U$, $\mathcal{C}^\infty$-smooth in $\overline{\Omega_P}\cap U$, and tangent to $b \Omega_P\cap U$, where $U$ is a neighborhood of $(0,0)\in \mathbb C^2$, then $H=i\beta z_2\frac{\partial}{\partial z_2}$ for some $\beta\in \mathbb R$. 
\end{theorem}

 In the case that the tangential holomorphic vector field $H$ is holomophic in a neighborhood of the origin, Theorem \ref{Th1} is already proved in \cite{By1, HN}. Here, since the tangential holomorphic vector field $H$ in Theorem \ref{Th1} is only holomorphic inside the domain, it seems to us that some key techniques in \cite{By1} could not use for our situation.  To get around this difficulty, we first employ the Schwarz reflection principle to show that the holomorphic functions $h_1, h_2$ must vanish to finite order at the origin. Then the equation (\ref{maineq}) implies that $h_1\equiv 0$. Therefore, from Chirka's curvilinear Hargtogs' lemma the proof finally follows (see the detailed proof in Section \ref{S2}).

We now note that $\mathrm{Aut}(\Omega_P)$ is noncompact since it contains biholomorphisms 
$$
(z_1,z_2)\mapsto (z_1+is, e^{it}z_2 ),~s,t\in \mathbb R.
$$
Let us denote by $\{R_t\}_{t\in \mathbb R}$ the one-parameter subgroup of $\mathrm{Aut}(\Omega_P,0)$ generated by the holomorphic vector field $H_R(z_1,z_2)=iz_2\frac{\partial }{\partial z_2}$, that is, 
$$
R_t(z_1,z_2)=\big(z_1,e^{it}z_2\big), ~\forall t\in\mathbb R.
$$
In addition, denote by $T_s(z_1,z_2)=(z_1+is,z_2)$ for $s\in \mathbb R$.

To state the second main result, we need the following definitions.
Recall that the Kobayashi metric $K_D$ of $D$ is defined by
$$K_D(\eta,X):=\inf\{\frac{1}{R}|\ \exists f: \Delta\to D~\text{such that} f(0)=\eta, f'(0)=R X\},$$
where $\eta\in D$ and $X\in T^{1,0}_\eta\mathbb C^n$, where $\Delta_r$ is a disc with center at the origin and radius $r>0$ and $\Delta:=\Delta_1$. 

The following definition derives from work of X. Huang (\cite{Hu95}).
\begin{define} Let $D$ be a domain in $\mathbb C^n$ with $\mathcal{C}^2$-smooth boundary $bD$ and $z_0$ be a boundary point.  For a $\mathcal{C}^1$-smooth monotonic increasing function $g: [1,+\infty)\to [1,+\infty)$, we say that $D$ is \emph{$g$-admissible} at $z_0$ if there exists a neighborhood $V$ of $z_0$ such that
$$
K_{D}(z,X)\gtrsim g(\delta_D^{-1}(z))|X|
$$
for any $z\in V\cap D$ and $X\in T^{1,0}_z\mathbb C^n$, where $\delta_D(z)$ is the distance of $z$ to $bD$.
\end{define}
\begin{remark}
\begin{itemize}
\item[(i)] It is proved in \cite[p.93]{Ber} (see also in \cite{Sib81}) that if there exists a plurisubharmonic peak function at $z_0$, then 
there exists a neighborhood $V$ of $z_0$ such that 
$$
K_D(z,X)\leq K_{D\cap V}(z,X)\leq  2K_D(z, X),
$$ 
for any $z\in V\cap D$ and $X\in T^{1,0}_z\mathbb C^n$.
\item[(ii)] If $D$ is $\mathcal{C}^\infty$-smooth pseudoconvex of finite type, then $D$ is $t^\epsilon$-admissible at any boundary point for some $\epsilon>0$ (cf. \cite{Cho92}). Recently, T. V. Khanh \cite{Khanh13} proved that a certain pseudoconvex domain of infinite type is also $g$-admissible for some function $g$.
\end{itemize}
\end{remark}

\begin{define}[see \cite{Khanh13}] Let $D\subset \mathbb C^n$ be a $\mathcal{C}^2$-smooth domain. Assume that $D$ is pseudoconvex near $z_0\in bD$. For a $\mathcal{C}^1$-smooth monotonic increasing function $u: [1,+\infty)\to [1,+\infty)$ with $u(t)/t^{1/2}$ decreasing, we say that a domain $D$ has the $u$-property at the boundary $z_0$ if there  exist a neighborhood $U$ of $z_0$ and a family of $\mathcal{C}^2$-functions $\{\phi_\eta\}$ such that
\begin{itemize}
\item[(i)] $|\phi_\eta|<1$, $\mathcal{C}^2$, and plurisubharmonic on $D$;
\item[(ii)] $i\partial\bar \partial \phi_\eta \gtrsim u(\eta^{-1})^2 Id$ and $|D\phi_\eta|\lesssim \eta^{-1}$ on $U\cap \{z\in D\colon -\eta<r(z)<0\}$, where $r$ is a $\mathcal{C}^2$-defining function for $D$.
\end{itemize}
Here and  in what follows, $\lesssim$ and $\gtrsim$ denote inequalities up to a positive constant multiple. In addition, we use $\approx $ for the combination of $\lesssim$ and $\gtrsim$.

\end{define}
\begin{define} [see \cite{Khanh13}]We say that a domain $D$ has the strong $u$-property at the boundary $z_0$ if it has the $u$-property with $u$ satisfying the following:
\begin{itemize}
\item[(i)] $\int\limits_t^{+\infty}\frac{da}{au(a)}$ for some $t>1$ and denote by $(g(t))^{-1}$ this finite integral;
\item[(ii)] The function $\frac{1}{\delta g\Big(1/\delta^\eta\Big)}$ is decreasing and $\int\limits_0^d\frac{1}{\delta g\Big(1/\delta^\eta)\Big)}d\delta<+\infty$ for $d>0$ small enough and for some $0<\eta<1$.
\end{itemize}
\end{define}
\begin{define} We say that $\Omega_P$ satisfies the condition (T) at $\infty$ if one of following conditions holds
\begin{itemize}
\item[(i)]$\lim_{z\to \infty}P(z)=+\infty$; 
\item[(ii)] The function $Q$ defined by setting $Q(\zeta):=P(1/\zeta)$ can be extended to be $\mathcal{C}^\infty$-smooth in a neighborhood of $\zeta=0$, $\Omega_Q$ has the strong $\tilde u$-property at $(-r,0)$ for some function $\tilde u$, where $r=\lim_{z\to \infty} P(z)$, and $b\Omega_P$ and $b\Omega_Q$ are not isomorphic as CR maniflod germs at $(0,0)$ and $(-r,0)$ respectively.
\end{itemize}

\end{define}

The second aim of this paper is to show the following theorem.
\begin{theorem} \label{Th2}
Let $P:\mathbb C\to \mathbb R$ be a $\mathcal{C}^\infty$-smooth function satisfying 
\begin{itemize}
\item[(i)] $P(z)=q(|z|)$ for all $z\in \mathbb C$, where $q:[0,+\infty)\to \mathbb R$ is a function with $q(0)=0$ such that it is strictly increasing and convex on $[0,\epsilon_0)$ for some $\epsilon_0>0$,
\item[(ii)] $P$ vanishes to infinite order at $0$, and
\item[(iii)] $P$ vanishes to finite order at any $z\in \mathbb C^*:=\mathbb C\setminus\{0\}$.
\end{itemize}
Assume that $\Omega_P$ has the strong $u$-property at $(0,0)$ and $\Omega_P$ satisfies the property (T) at $\infty$. Then 
$$
\mathrm{Aut}(\Omega_P)=\{(z_1,z_2)\mapsto (z_1+is, e^{it}z_2 )\colon s,t \in \mathbb R\}.
$$
\end{theorem}

\begin{remark}\label{remark}
Let $\Omega_P$ be as in Theorem \ref{Th2} and let $P_\infty(b \Omega_P)$ the set of all points in $b \Omega_P$ of D'Angelo infinite type. It is easy to see that $P_\infty(b \Omega_P)=\{(it,0)\colon t\in \mathbb R\}$. Moreover, since $\Omega_P$ is invariant under any translation $(z_1,z_2)\mapsto (z_1+it,z_2),~t\in \mathbb R$, it satisfies the trong $u$-property at $(it,0)$ for any $t\in \mathbb R$.
\end{remark}
\begin{remark}\label{remark4}
 Let $P$ be a function defined by $P(z_2)=\exp(-1/|z_2|^\alpha)$ if $z_2\ne 0$ and $P(0)=0$, where $0<\alpha<1$. Then by \cite[Corollary 1.3]{Khanh13}, $\Omega_P$ has $\log^{1/\alpha}$-property at $(it,0)$ and thus it is $\log^{1/\alpha-1}$-admissible at $(it,0)$ for any $t\in \mathbb R$. Furthermore, a computation shows that  if $0<\alpha<1/2$, then $\Omega_P$ has the strong $\log^{1/\alpha}$-property at $(it,0)$ for any $t\in \mathbb R$. 
\end{remark}
\begin{example}\label{ex1}
Let $E_j$, $j=1,\ldots,3$, be domains defined by
$$
E_j:=\{(z_1,z_2)\in \mathbb C^2\colon \rho(z_1,z_2)= \mathrm{Re}~z_1+P_j(z_2)<0\},
$$
where $P_j$ are defined by
\begin{equation*}
\begin{split}
P_1&=\psi(|z_2|)e^{-1/|z_2|^{\alpha}}+(1-\psi(|z_2|))\frac{1}{|z_2|^{2m}},\\
P_2&=\psi(|z_2|)e^{-1/|z_2|^{\alpha}}+(1-\psi(|z_2|))e^{-1/|z_2|^{\beta}},\\
P_2&=\psi(|z_2|)e^{-1/|z_2|^{\alpha}}+(1-\psi(|z_2|))|z_2|^2
\end{split}
\end{equation*}
if $z_2\ne 0$ and $P(0)=0$, where $0<\alpha,\beta<1/2, m\in \mathbb N^*)$ with $\beta\ne \alpha$ and $\psi(t)$ is a $\mathcal{C}^\infty$-smooth cut-off function such that $\psi(t)=1$ if $|t|<a$ and $\psi(t)=0$ if $|t|>b~(0<a<b)$. It follows from Remark \ref{remark4} and a computation that $E_j,~j=1,\ldots,3,$ have the strong $\log^{1/\alpha}$-property and satisfy the property (T) at $\infty$. Therefore, by Theorem \ref{Th2} we conclude that
$$
\mathrm{Aut}(E_j)=\{(z_1,z_2)\mapsto (z_1+is, e^{it}z_2 )\colon s,t \in \mathbb R\},~j=1,\ldots,3.
$$
\end{example}

We explain now the idea of proof of Theorem \ref{Th2}. Let $f\in \mathrm{Aut}(\Omega_P)$ be an arbitrary. We show that there exist $t_1,t_2\in \mathbb R$ such that $f,f^{-1}$ extend smoothly to $b\Omega_P$ near $(it_1,0)$ and $(it_2,0)$ respectivey and $(it_2,0)=f(it_1,0)$ (cf. Lemma \ref{lem6}). Replacing $f$ by $T_{-t_2}\circ f\circ T_{t_1}$, we may assume that $f,f^{-1}$ extend smoothly to $b\Omega_P$ near the origin and $f(0,0)=(0,0)$. Next, we consider the one-parameter subgroup $\{F_t\}_{t\in \mathbb R}$ of $\mathrm{Aut}(\Omega_P)\cap \mathcal{C}^\infty(\overline{\Omega_P}\cap U)$ defined by $F_t=f\circ R_{-t}\circ f^{-1}$. By employing Theorem \ref{Th1}, there exists a real number $\delta$ such that $F_t=R_{\delta t}$ for all $t\in\mathbb R$. Using the property that $P$ vanishes to infinite order at $0$, it is proved that $f=R_{t_0}$ for some $t_0\in \mathbb R$ (see the detailed proof in Section $4$). This finishes our proof.

This paper is organized as follows. In Section $2$,  we prove Theorem \ref{Th1}. In Section $3$,  we prove several lemmas to be used mainly in the proof of Theorem \ref{Th2}. Section $4$ is devoted to the proof of Theorem \ref{Th2}. Finally, two lemma are given in Appendix.


\section{Holomorphic vector fields tangent to an infinite type model}\label{S2}

This section is devoted to the proof of Theorem \ref{Th1}. Assume that $P:\mathbb C\to \mathbb R$ is a $\mathcal{C}^\infty$-smooth function satisfying (i) and (ii) as in Introduction. 

Then we consider a nontrivial holomorphic vector field $H=h_1(z_1,z_2)\frac{\partial}{\partial z_1}+h_2(z_1,z_2)\frac{\partial}{\partial z_2}$ defined on $\Omega_P\cap U$, where $U$ is a neighborhood of the origin. We only consider $H$ is tangent to $b \Omega_P\cap U$. This means that they satisfy the identity
\begin{equation}\label{eq1}
(\mathrm{Re}~H)\rho(z_1,z_2)=0,~\forall~ (z_1,z_2)\in b \Omega_P\cap U. 
\end{equation}
By a simple computation, we have 
\begin{equation*}
\begin{split}
 \rho_{z_1}(z_1,z_2)&= 1,\\
\rho_{z_2}(z_1,z_2)&= P'(z_2),
\end{split}
\end{equation*}
and the equation (\ref{eq1}) can thus be rewritten as
\begin{equation}\label{eq2}
\begin{split}
&\mathrm{Re} \Big[h_1(z_1,z_2)+P'(z_2)h_2(z_1,z_2)\Big ]=0
\end{split}
\end{equation}
for all $(z_1,z_2)\in b \Omega_P\cap U$.

Since $\Big(it-P(z_2), z_2\Big)\in b \Omega_P$ for any $t \in \mathbb R$ with $t$ small enough, the above equation again admits a new form
\begin{equation}\label{eq3}
\begin{split}
&\mathrm{Re}\Big[ h_1\big(it-P(z_2),z_2\big)+P'(z_2)h_2\big(it-P(z_2),z_2\big)\Big]=0
\end{split}
\end{equation}
for all $z_2\in\mathbb C$ and for all $t\in\mathbb R$ with $|z_2|<\epsilon_0$ and $|t|<\delta_0$, where $\epsilon_0>0$ and $\delta_0>0$ are small enough.

\begin{lemma}\label{lem6} We have that
 $\frac{\partial^{m+n}}{\partial z_1^m\partial z_2^n} h_1(z_1,0)$ can be extended to be holomorphic in a neighborhood of $z_1=0$ for every $m,n\in\mathbb N$. 
\end{lemma}
\begin{proof}
Since $\nu_0(P')=+\infty$, it follows from (\ref{eq3}) with $t=0$ that $\mathrm{Re} h_1(it,0)=0$ for all $t\in (-\delta_0,\delta_0)$. By the Schwarz reflection principle, $h_1(z_1,0)$ can be extended to a holomorphic function on a neighborhood of $z_1=0$. For any $m,n\in\mathbb N$, taking $\frac{\partial^{m+n}}{\partial t^m\partial z_2^n}\mid_{z_2=0}$ of both sides of the equation (\ref{eq3}) one has
$$
 \mathrm{Re}\Big[i^m \frac{\partial^{m+n}}{\partial z_1^m\partial z_2^n} h_1(it,0)\Big]=0
$$
for all $t\in (-\delta_0,\delta_0)$. Again by the Schwarz reflection principle, $\frac{\partial^{m+n}}{\partial z_1^m\partial z_2^n} h_1(z_1,0)$ can be extended to be holomorphic in a neighborhood of $z_1=0$, which completes the proof.
\end{proof}
\begin{corollary} \label{coro2}
If $h_1$ vanishes to infinite order at $(0,0)$, then $h_1\equiv 0$.
\end{corollary}
\begin{proof} 
Since $h_1$ vanishes to infinite order at $(0,0)$, $\frac{\partial^{m+n}}{\partial z_1^m\partial z_2^n} h_1(z_1,0)$ also vanishes to infinite at $z_1=0$ for all $m,n\in\mathbb N$. Moreover, by Lemma \ref{lem6} these functions are holomorphic in a neighborhood of $z_1=0$. Therefore, $\frac{\partial^{m+n}}{\partial z_1^m\partial z_2^n} h_1(z_1,0)\equiv 0$ for every $m,n\in\mathbb N$.  

Expand $h_1$ into the Taylor series at $(-\epsilon,0)$ with $\epsilon>0$ small enough so that  
$$
h_1(z_1,z_2)=\sum_{m,n=0}^\infty\frac{1}{m!n!} \frac{\partial^{m+n}}{\partial z_1^m\partial z_2^n} h_1(-\epsilon,0)(z_1+\epsilon)^m z_2^n 
$$ 

Since $\frac{\partial^{m+n}}{\partial z_1^m\partial z_2^n} h_1(-\epsilon,0)=0$ for all $m,n\in\mathbb N$, $h_1\equiv 0$ on a neighborhood of $(-\epsilon,0)$, and thus $h_1\equiv 0$ on $\Omega_P$.  

\end{proof}

\begin{proof}[Proof of Theorem \ref{Th1}]

Denote by $D_P(r):=\{z_2\in \mathbb C: |z_2|< q^{-1}(r) \}~(r>0)$. For each $z_1$ with $\mathrm{Re}(z_1)<0$, we have
\begin{equation}\label{eq4}
h_1(z_1,z_2)=\sum_{n=0}^\infty a_n(z_1)z_2^n, ~\forall ~ z_2\in D_P(-\mathrm{Re}(z_1)),
\end{equation}
where $a_n(z_1)=\frac{\partial^n}{\partial z_2^n}h_1(z_1,0)$ for every $n\in \mathbb N$. Since $h_1\in \mathrm{Hol}(\Omega_P\cap U)\cap \mathcal{C^\infty}(\overline{\Omega_P}\cap U)$, $a_n\in Hol(\mathcal{H}\cap U_1)\cap \mathcal{C}^\infty(\overline{\mathcal{H}}\cap U_1)$ for every $n=0,1,\ldots$, where $\mathcal{H}:=\{z_1\in \mathbb C\colon \mathrm{Re}(z_1)<0\}$ and $U_1$ is a neighborhood of $z_1=0$ in $\mathbb C_{z_1}$. Moreover, expanding the function $g_{z_1}(z_2):=h_1(z_1,z_2)$ into the Fourier series we can see that (\ref{eq4}) still holds for all $z_2\in \overline{D_P(-\mathrm{Re}(z_1))}$. Therefore, the function $h_1(it-P(z_2),z_2)$ can be rewritten as follows:
$$
h_1(it-P(z_2),z_2)=\sum_{n=0}^\infty a_n(it-P(z_2))z_2^n,
$$
for all $(t,z_2)\in (-\delta_0,\delta_0)\times \Delta_{\epsilon_0}$, where $\delta_0>0, \epsilon_0>0$ are small enough.

Similarly, we also have
$$
h_2(it-P(z_2),z_2)=\sum_{n=0}^\infty b_n(it-P(z_2))z_2^n
$$
for all $(t,z_2)\in (-\delta_0,\delta_0)\times \Delta_{\epsilon_0}$, where $b_n \in Hol(\mathcal{H}\cap U_1)\cap \mathcal{C}^\infty(\overline{\mathcal{H}}\cap U_1)$ for every $n=0,1,\ldots$.

Now we shall prove that $h_1\equiv 0$. Indeed, aiming for contradiction, we suppose that $h_1\not\equiv 0$. If $h_1$ vanishes to infinite order at $(0,0)$, then by Corollary \ref{coro2} one gets $h_1\equiv 0$. So, $h_1$ vanishes to finite order at $(0,0)$. It follows from (\ref{eq3}) that $h_2$ also vanishes to finite order at $(0,0)$, for otherwise $h_1$ vanishes to infinite order at $(0,0)$.

Denote by
\begin{equation}\label{eq6}
\begin{split}
m_0&:=\min\Big\{m\in \mathbb N \colon \frac{\partial^{m+n}}{\partial^m z_1 \partial^n z_2}h_1(0,0)\ne 0~\text{for some}~ n\in\mathbb N\Big\}, \\
n_0&:=\min\Big\{n\in \mathbb N \colon \frac{\partial^{m_0+n}}{\partial^{m_0} z_1 \partial^n z_2}h_1(0,0)\ne 0\Big\}, \\
k_0&:=\min\Big\{m\in \mathbb N \colon \frac{\partial^{k+l}}{\partial^k z_1 \partial^l z_2}h_2(0,0)\ne 0~\text{for some}~ l\in\mathbb N\Big\}, \\
l_0&:=\min\Big\{l\in \mathbb N \colon \frac{\partial^{k_0+l}}{\partial^{k_0} z_1 \partial^l z_2}h_2(0,0)\ne 0\Big\}.
\end{split}
\end{equation}

Since $\nu_0(P)=+\infty$, one obtains that
\begin{equation}\label{eq6}
\begin{split}
h_1(i\alpha P(z_2)-P(z_2),z_2)&= a_{m_0,n_0} (i\alpha-1)^{m_0}(P(z_2))^{m_0}\big(z_2^{n_0}+o(|z_2|^{n_0}\big),\\
h_2(i\alpha P(z_2)-P(z_2),z_2)&= b_{k_0,l_0} (i\alpha-1)^{k_0}(P(z_2))^{k_0}\big(z_2^{l_0}+o(|z_2|^{l_0}\big),
\end{split}
\end{equation}
where $a_{m_0,n_0}:=\frac{\partial^{m_0+n_0}}{\partial^{m_0} z_1 \partial^{n_0} z_2}h_1(0,0)\ne 0,~b_{k_0,l_0}:=\frac{\partial^{k_0+l_0}}{\partial^{k_0} z_{l_0} \partial^{l_0} z_2}h_2(0,0) \ne 0$, and $\alpha\in\mathbb R$ will be chosen later.

Now it follows from (\ref{eq3}) with $t=\alpha P(z_2)$ that 
\begin{equation}\label{eq7}
\begin{split}
\mathrm{Re} &\Big[ a_{m_0 n_0}(i\alpha -1)^{m_0}(P(z_2))^{m_0}\big(z_2^{n_0}+o(|z_2|^{n_0})\big)+
 b_{k_0 l_0}(i\alpha -1)^{k_0}(z_2^{l_0}+o(|z_2|^{l_0}) \\
&\times (P(z_2))^{k_0} P'(z_2)   \Big ]=0
\end{split}
 \end{equation}
for all $z_2\in \Delta_{\epsilon_0}$ and for all $\alpha \in \mathbb R$ small enough. We note that in the case $n_0=0$ and $\mathrm{Re}(a_{m_0 0})=0$, $\alpha$ can be chosen in such a way that $\mathrm{Re}\big( (i\alpha-1)^{m_0}a_{m_0 0}\big)\ne 0$. Then the above equation yields that $k_0>m_0$. Furthermore, since $P$ is rotational, it follows that $\mathrm{Re}(iz_2P'(z_2))\equiv 0$ (see \cite[Lemma $4$]{Kim-Ninh}), and hence we can assume that $\mathrm{Re}(b_{10})\ne 0$ for the case that $k_0=1,l_0=0$. However, (\ref{eq7}) contradicts Lemma $3$ in \cite{Kim-Ninh}. Therefore, $h_1\equiv 0$.

Granted $h_1\equiv 0$, (\ref{eq3}) is equivalent to
\begin{equation}\label{eq8}
\begin{split}
\mathrm{Re} \Big[ P'(z_2)h_2(it-P(z_2),z_2)   \Big ]=0
\end{split}
 \end{equation}
for all $(t,z_2)\in (-\delta_0,\delta_0)\times \Delta_{\epsilon_0}$. Thus, for each $z_2\in \Delta^*_{\epsilon_0}$ the function $g_{z_2}$ defined by setting $g_{z_2}(z_1):=h_2(z_1,z_2)$ is holomorphic in $\{z_1\in \mathbb C:\mathrm{Re}(z_1)<-P(z_2) \}$ and $\mathcal{C}^\infty$-smooth up to the real line $\{z_1\in \mathbb C: \mathrm{Re}(z_1)=-P(z_2)\}$. Moreover, $g_{z_2}$ maps this line onto the real line $\mathrm{Re}(P'(z_2)w)=0$ in the complex plane $\mathbb C_w$. Thus, by the Schwarz reflection principle, $g_w$ can be extended to be holomorphic in a neighborhood $U$ of $z_1=0$ in the plane $\mathbb C_{z_1}$. (The neighborhood $U$ is independent of $z_2$.)  

Now our function $h_2$ is holomorphic in $z_1\in U$ for each $z_2\in\Delta^*_{\epsilon_0}$ and holomorphic in $(z_1,z_2)$ in the domain $\{(z_1,z_2)\in \mathbb C^2: \mathrm{Re}(z_1)<0,~ |z_2|<q^{-1}(-\mathrm{Re}(z_1))\}$. Therefore, it follows from Chirka's curvilinear Hartogs' lemma (see \cite{Ch}) that $h_2$ can be extended to be holomorphic in a neighborhood of $(0,0)$ in $\mathbb C^2$. Moreover, by (\ref{eq8}) and by \cite[Theorem $3$]{HN} we conclude that $h_2(z_1,z_2)\equiv i\beta z_2$ for some $\beta\in\mathbb R^*$. So, the proof is complete.

\end{proof}
\section{Extension of automorphisms }

If $f: D\to \mathbb C^N$ is a continuous map on a domain $D\subset \mathbb C^n$  and $z_0\in \partial D$, we denote by $\mathcal{C}(f,z_0)$ the cluster set of $f$ at $z_0$:
$$
\mathcal{C}(f,z_0)=\{w\in \mathbb C^N \colon w=\lim f(z_j), z_j\in D, \text{and}~\lim z_j=z_0 \}.
$$
\begin{define}[see \cite{Bel81}]
When $\Gamma$ be an open subset of the boundary of a smooth domain $D$, we say that $\Gamma$ satisfies \emph{local condition} $\mathrm{R}$ if for each $z\in \Gamma$, there is an open set $V$ in $\mathbb{C}^n$ with $z\in V$ such that for each $s$, there is an $M$ such that 
$$
\mathrm{P}\big(W^{s+M}(D\cap V)\big)\subset W^{s}(D\cap V).
$$
We say that $D$ satisfies \emph{local condition} $\mathrm{R}$ at $z_0\in bD$ if there exists an open subset of the boundary $bD$ containing $z_0$ and satisfying local condition $\mathrm{R}$.
\end{define}

\begin{define}
Let $D, G$ be domains in $\mathbb C^n$ and let $F: [0,+\infty)\to [0,+\infty)$ be an inceasing function with $F(0)=0$. Let $z_0\in bD$ and $w_0\in bG$. We say that $D,G$ satisfies the property $(D,G)_{(z_0,w_0)}^F$ if for each proper holomorphic mapping $f: D\to G$, there exist neighborhoods $U$ and $V$ of $z_0$ and $w_0$ respectively such that
$$
d_G(f(z))\leq  F(d_D(z))
$$
for any $z\in U\cap D$ such that $f(z)\in V\cap G$.
\end{define}
For the case $D$ and $D$ are bounded pseudoconvex domains with generic corners, D. Chakrabarti and K. Verma \cite[Propsition 5.1]{CV14} proved there exists a $\delta\in (0,1)$ such that 
$$
(d_D(z))^{1/\delta}\lesssim d_G(f(z))\lesssim   (d_D(z))^\delta
$$
for all $z\in D$, which is a generalization of  \cite{DF79,Ber91}. Consequently, $D,G$ satisfies the property $(D,G)_{(z_0,w_0)}^F$, where $F(t)=t^\delta$, for any $z_0\in bD$ and $w_0\in bG$. 
 
We now recall the general H\"{o}lder continuity (see \cite{Khanh13}). Let $f$ be an increasing function such that $\lim_{t\to +\infty}f(t)=+\infty$. For $\Omega\subset \mathbb C^n$, define $f$-H\"{o}lder space on $\bar \Omega$ by
$$
\Lambda^f(\overline{\Omega})=\{u\colon \|u\|_\infty+\sup_{z,z+h\in \bar\Omega}f(|h|^{-1})|u(z+h)-u(z)|<+\infty\}.
$$
Note that the $f$-H\"{o}lder space includes the standard H\"{o}lder space $\Lambda_\alpha(\overline{\Omega})$ by taking $f(t)=t^\alpha$ with $0<\alpha<1$.

The following lemma is a slight generalization of \cite[Theorem 1.4]{Khanh13}. 
\begin{lemma} \label{lem2}
 Let $D$ and $G$ be domains in $\mathbb C^n$ with $\mathcal{C}^2$-smooth boundaries. Let $g:[1,+\infty)\to [1,+\infty)$ and $F: [0,+\infty)\to [0,+\infty)$ be nonnegative increasing functions with $F(0)=0$ such that the function $\frac{1}{\delta g\Big(1/F\big(\delta\big)\Big)}$ is decreasing and $\int\limits_0^d\frac{1}{\delta g\Big(1/F\big(\delta\big)\Big)}d\delta<+\infty$ for $d>0$ small enough. Assume that $D$ and $G$ satisfies the property $(D,G)_{(z_0,w_0)}^F$ and $G$ is $g$-admissible at $w_0$. Let $f:D\to G$ be a proper map such that $w_0\in \mathcal{C}(f,z_0)$. Then there exist neighborhoods $U$ and $V$ of $z_0$ and $w_0$, respectively, such that $f$ can be extended as a general H\"{o}der continuous map $\hat f: U\cap \overline{D}\to V\cap \overline{G}$ with a rate $h(t)$ defined by
$$
(h(t))^{-1}:=\int^{t^{-1}}_0\frac{1}{\delta g\Big(1/F\big(\delta\big)\Big)}d\delta.
$$
\end{lemma}
\begin{proof}
Since $G$ is $g$-admissible at $w_0$, using the Schwarz-Pick lemma for the Kobayashi metric and the upper bound of Kobayashi metric, we obtain the following estimate
$$
g(\delta^{-1}_G(f(z)))|f'(z)X|\lesssim K_{G}(f(z),f'(z)X)\leq K_D(z,X)\lesssim \delta^{-1}_D(z) |X|
$$
for any $z\in D\cap  U$ such that $f(z)\in V\cap G$ and $X\in T^{1,0}\mathbb C^n$. Moreover, since the property $(D,G)_{(z_0,w_0)}^F$ holds, we may assume that
$$
\delta_{G}(f(z))\leq F(\delta_D(z))
$$
for any $z\in D\cap  U$ such that $f(z)\in V\cap G$. Therefore,
\begin{equation}\label{22-10}
|f'(z)X|\lesssim \frac{1}{\delta_D(z)g\Big(1/F\big(\delta_D(z)\big)\Big)} |X|
\end{equation}
for any $z\in D\cap U$ such that $f(z)\in V\cap G$ and $X\in T^{1,0}\mathbb C^n$. 

By using the Henkin's technique (see \cite{Ber92, Suk}), we are going to prove that $f$ extends continuously to $z_0$. Indeed, suppose that $f$ does not extend continuously to $z_0$: there are an open ball $B\subset V$ (with center at $w_0$) and a neighborhoods basis $U_j$ of $z_0$ such that $f(D\cap U_j)$ is connected and never contained in $B$. Then, since $w_0\in \mathcal{C}_\Omega(f,z_0)$, there exists a sequence $\{z_j'\}$, $z_j'\in U_j$ such that $f(z_j')\in \partial B$ and $\lim f(z_j')=w_0'\in \partial B\cap\partial G$. 

Let $\{z_j\}\subset \Omega\cap U$ such that $\lim z_j=z_0$ and $\lim f(z_j)=w_0$. Let $l_j:=|z_j-z_j'|$ and $\gamma_j:[0,3l_j]\to \Omega\cap U$ be a $\mathcal{C}^1$-path such that:
\begin{itemize}
\item[(a)] $\gamma_j(0)=z_j$ and $\gamma_j(3l_j)=z_j'$.
\item[(b)] $\delta_\Omega(\gamma(t))\geq t$ on $[0,l_j]$; $\delta_\Omega(\gamma(t))\geq l_j$ on $[l_j,2l_j]$; $\delta_\Omega(\gamma(t))\geq 3l_j-t$ on $[2l_j,3l_j]$.
\item[(c)] $\|\frac{d\gamma_j(t)}{dt}\|\lesssim 1$, $t\in [0,3l_j]$.
\end{itemize}
(See \cite[Prop. 2, p. 203]{HL84}).

Choose $t_j\in [0,3l_j]$ such that $f\circ \gamma_j([0,t_j])\subset \bar B$ and $f\circ \gamma_j(t_j)\in \partial B$. It follows from (\ref{22-10}), (b) and (c) that $|f(z_j)-f\circ \gamma_j(t_j)|\lesssim 1/h(1/l_j)+1/g\Big(1/F\big(l_j\big)\Big)\to 0$ as $j\to \infty$: a contradiction. Hence, $f$ extends continuously to $z_0$.

We may now assume that $f(D\cap U)\subset G\cap V$ and apply \cite[Lemma 1.4]{Khanh13} for proving that $f$ can be extended to a $h$-H\"{o}der continuous map $\hat f: \overline{D}\cap  U\to \overline{G}\cap V$ with the rate $h(t)$ defined by
$$
((h(t))^{-1}:=\int^{t^{-1}}_0 \frac{1}{\delta g\Big(1/F\big(\delta\big)\Big)} d\delta.
$$
\end{proof}

The following lemma is a local version of Fefferman's theorem (see \cite{Bel81}).
\begin{lemma}\label{lem3}
Suppose that $D$ and $G$ are $\mathcal{C}^\infty$-smooth domains in $\mathbb  C^n$ satisfying local condition $\mathrm{R}$ at $z_0\in bD$ and $w_0\in bG$ respectively. Assume that $D$ and $G$ are pseudoconvex near $z_0$ and $w_0$ respectively. Let $g:[1,+\infty)\to [1,+\infty)$ and $F: [0,+\infty)\to [0,+\infty)$ be nonnegative increasing functions with $F(0)=0$ such that the function $\frac{1}{\delta g\Big(1/F\big(\delta\big)\Big)}$ is decreasing and $\int\limits_0^d\frac{1}{\delta g\Big(1/F\big(\delta\big)\Big)}d\delta<+\infty$ for $d>0$ small enough. Suppose that $D$ and $G$ satisfy the property $(D,G)^F_{(z_0,w_0)}$. Let $f$ be a biholomorphic mapping of $D$ onto $G$ such that $w_0\in \mathcal{C}(f,z_0)$. Then $f$ extends smoothly to $bD$ in some neighborhood of the point $z_0$.
\end{lemma}
\begin{proof}
By Lemma \ref{lem2}, we may assume that there exist neighborhoods $U$ and $V$ of $z_0$ and $w_0$ respectively such that $f$  extends continuously to $U\cap \bar D$. Moreover, we may assume that $f(U\cap D)=V\cap G$ and $U\cap D$ and $V\cap G$ are  bounded $\mathcal{C}^\infty$-smooth pseudoconvex domains. Therefore, the proof follows from the theorem in \cite[Section 7]{Bel81}.
\end{proof}
\begin{lemma}\label{lem4}
Let $D \subset \mathbb C^n$ be a $\mathcal{C}^2$-smooth domain and let $0<\eta<1$. Assume that $D$ is pseudoconvex near $z_0\in bD$ and $D$ has $u$-property at $z_0$, where  $u: [1,+\infty)\to [1,+\infty)$ is a smooth monotonic increasing function with $u(t)/t^{1/2}$ decreasing and $\int\limits_{t_0}^{+\infty} \frac{da}{au(a)}<+\infty$ for some $t_0>1$. Then $D$ is  $g$-admissible at $z_0$, where $g$ is a function defined by
$$
(g(t))^{-1}=\int_t^{+\infty} \frac{da}{au(a)},~ t_0\leq t<+\infty.
$$
Moreover, the property $(D,G)_{(z_0,w_0)}^{F_2}$ holds for any $\mathcal{C}^2$-smooth domain $G\subset \mathbb C^n$ and $w_0\in bG$, where $F_2(t):=c_2 t^\eta,~t>0$, for some $c_2>0$.
\end{lemma}
\begin{proof} Let $D \subset \mathbb C^n$ be a $\mathcal{C}^2$-smooth domain. Assume that $D$ is pseudoconvex near $z_0\in bD$ and $D$ has the $u$-property at $z_0$, where  $u: [1,+\infty)\to [1,+\infty)$ is a smooth monotonic increasing function with $u(t)/t^{1/2}$ decreasing and $\int_{t_0}^{+\infty} \frac{da}{au(a)}<+\infty$ for some $t_0>1$. It follows from \cite[Theorem 1.2]{Khanh13} that $D$ is  $g$-admissible at $z_0$, where $g$ is a function defined by
$$
(g(t))^{-1}=\int_t^{+\infty} \frac{da}{au(a)},~ t_0\leq t<+\infty.
$$
Denote by $\tilde g$ the functions defined by
\begin{equation*}
\begin{split}
\tilde g(\delta)=\frac{1}{g^{-1}(1/(\gamma \delta))},
\end{split}
\end{equation*}
for any $0<\delta<\delta_0$, where $\gamma, \delta_0$ sufficiently small. By \cite[Theorem 3.1]{Khanh13} and the proof of \cite[Theorem 2.1]{Khanh13}, there exists a family $\psi_w(z)$ as in Lemma \ref{lemma12} in Appendix, where $F_1:=c_1 \tilde{g}^\eta$ and $F_2(t):=c_2 t^\eta,~t>0$, for some $0<\eta<1$ and $c_1,c_2>0$. Therefore, it follows from Lemma \ref{lemma12} in Appendix that the property $(D,G)_{(z_0,w_0)}^{F_2}$ holds for any $\mathcal{C}^2$-smooth domain $G\subset \mathbb C^n$ and $w_0\in bG$. This finishes the proof.
\end{proof}
By the definition of strong $u$-property, lemmas \ref{lem3} and \ref{lem4}, we obtain the following corollary.
\begin{corollary}\label{Co2}
Suppose that $D$ and $G$ are $\mathcal{C}^\infty$-smooth domains in $\mathbb  C^n$ satisfying the local condition R at $z_0\in bD$ and $w_0\in bG$ respectively. Suppose that $D$ and $G$ are pseudoconex near $z_0$ and $w_0$ respectively. Assume that $D$ (resp. $G$) has the strong $u$-property at $z_0$ (resp. strong $\tilde u$-property at $w_0$).
Let $f$ be a biholomorphic mapping of $D$ onto $G$ such that $w_0\in \mathcal{C}(f,z_0)$. Then $f$ and $f^{-1}$ extend smoothly to $bD$ in some neighborhoods of the points $z_0$ and $w_0$, respectively.
\end{corollary}
\begin{remark}\label{Re5} Suppose that $D$ is $\mathcal{C}^\infty$-smooth pseudoconvex of finite type near $z_0\in bD$. Then  $D$ has the $t^\epsilon$-property at $z_0$ for some $\epsilon>0$ (cf. \cite{Cho92, Khanh13}). Moreover, a computation shows that the strong $t^\epsilon$-property at $z_0$. In addition, $D$ satisfies the local condition R at $z_0$ (cf. \cite{BL80}).
\end{remark}

By Corollary \ref{Co2} and Remark \ref{Re5}, we obtain the following corollary which is proved by A. Sukhov. 
\begin{corollary}[See Corollary 1.4 in \cite{Suk}]\label{Cor2}
Suppose that $D$ and $G$ are $\mathcal{C}^\infty$-smooth domains in $\mathbb  C^n$. Suppose that $D$ and $G$ are pseudoconex of finite type near $z_0\in bD$ and $w_0\in bG$ respectively. Let $f$ be a biholomorphic mapping of $D$ onto $G$ such that $w_0\in \mathcal{C}(f,z_0)$. Then $f$ and $f^{-1}$ extend smoothly to $bD$ in some neighborhoods of the points $z_0$ and $w_0$, respectively.
\end{corollary}

It is well-known that any accumulation orbit boundary point is pseudoconvex (cf. \cite{GK91}). The following lemma says that the pseudoconvexity is invariant under any biholomorphism.
\begin{lemma}\label{le3}
 Let $D,G$ be $\mathcal{C}^2$-smooth domains in $\mathbb C^n$ and let $z_0\in bD$ and $w_0\in bG$. Let $f:D\to G$ be a biholomorphism such that $w_0\in \mathcal{C}(f,z_0)$. If $D$ is pseudoconvex at $z_0$, then $G$ is also pseudoconvex at $w_0$.
\end{lemma}
\begin{proof}
Since $w_0\in \mathcal{C}(f,z_0)$, we may assume that there exists a sequence $\{z_j\}\subset D$ such that $z_j\to z_0$ and $f(z_j)\to w_0$ as $j\to \infty$. Assume the contrary, that $G$ is not pseudoconvex at $w_0$. Then there is a compact set $K\Subset G$ such that the holomorpphic hull $\hat{K}$ of $K$ contains $V\cap G$, where $V$ is a small neighborhood of $w_0$. (Recall that $\hat{K}:=\{z\in G\colon |g(z)|\leq \max_{K}|g|,~\forall~g:G\to \mathbb C~\text{holomorphic} \}$.)  Consequently, $f(z_j)\in\hat{K}$ for every $j\geq j_0$, where $j_0$ is big enough.

Denote by $L:=f^{-1}(K)$. Then $L$ is a compact subset in $D$. We shall prove that $z_j\in \hat{L}$ for every $j\geq j_0$ and hence the proof follows. Indeed,
let $g:D\to \mathbb C$ be any holomorphic function. Then since $f(z_j)\in \hat{K}$ for every $j\geq j_0$, we have
$$
|g\circ f^{-1}(f(z_j))|\leq \max_{K} |g\circ f^{-1}|,~\forall~j\geq j_0.
$$
This implies that
$$
|g(z_j)|\leq \max_{K} |g\circ f^{-1}|=\max_{f^{-1}(K)}|g|=\max_{L}|g|,~\forall~j\geq j_0.
$$
Therefore, $z_j\in \hat{L}$ for every $j\geq j_0$, and thus the proof is complete.
\end{proof}
\begin{lemma}\label{lem6} Let $\Omega_P$ be as in Theorem \ref{Th2} and let $f\in \mathrm{Aut}(\Omega_P)$ be arbitrary. Then there exist $t_1,t_2\in\mathbb R$ such that $f$ and $f^{-1}$ extend to be locally $\mathcal{C}^\infty$-smooth up to the boundaries near $(it_1,0)$ and $(it_2,0)$, respectively, and $f(it_1,0)=(it_2,0)$. 
\end{lemma}
\begin{proof}
We shall follow the proof of \cite[Lemma $3.2$]{Ber94}. Let $\phi: \Omega_P \to \Delta $ be the function defined by 
$$
\phi(z_1,z_2)=\frac{z_1+1}{z_1-1}.
$$
Then we see that $\phi$ is continuous on $\overline{\Omega_P}$ such that $|\phi(z)|<1$ for $z=(z_1,z_2)\in \Omega_P$ and tends to $1$ when $z_1 \to \infty$. Let $f: \Omega_P\to \Omega_P$ be an automorphism. We claim that there exists $t_1\in \mathbb R$ such that $\lim_{x\to 0^-} \inf |\pi_1 \circ f (x+it_1,0)|<+\infty$. Here, $\pi_1,\pi_2$ are the projections of $\mathbb C^2$ onto $\mathbb C_{z_1}$ and $\mathbb C_{z_2}$, respectively, i.e. $\pi_1(z)=z_1$ and $\pi_2(z)=z_2$. Indeed, if this would not be the case, the function $\phi\circ f$ would be equal to $1$ on the half plane $\{(z_1,z_2)\in \mathbb C^2\colon\mathrm{Re}~z_1<0, z_2=0\}$ and this is impossible since $|\phi(z)|<1$ for every $z\in \Omega_P$. Therefore, we may assume that there exists a sequence $x_k<0$ such that $\lim_{k\to \infty} x_k=0$ and $\lim_{k\to \infty } \pi_1\circ f (x_k+it_1,0)=w_1^0\in \overline{\mathcal{H}}$.

We shall prove that, after taking some subsequence if necessary, $\lim_{k\to \infty}\pi_2\circ f (x_k+it_1,0)=w_2^0$ for some $w_2^0\in \mathbb C$. Indeed, arguing by contradiction we assume that  $\pi_2\circ f (x_k+it_1,0)\to \infty$ as $k\to \infty$. Because of the convergence of $\{\pi_1\circ f (x_k+it_1,0)\}$, the sequence $\{P(\pi_2\circ f (x_k+it_1,0))\}$ is bounded, which is a contradition if $\lim_{z_2\to \infty}P(z_2)=+\infty$. Therefore, after taking some subsequence if necessary, we may assume that 
$$
 \lim_{k\to \infty} P(\pi_2\circ f (x_k+it_1,0))=r\geq 0.  
$$

Define $\psi(w_1,w_2)=(w_1,1/w_2)$. Then the map $\psi\circ f$ is well-defined near $(it_1,0)$ and 
$$
 \lim_{k\to \infty } \psi\circ f (x_k+it_1,0)=(w_1^0,0).
$$
Moreover, the defining function for $\psi\circ f(\Omega_P\cap U)$ near $(w_1^0,0)$, where $U$ is a small neighborhood of $(it_1,0)$, is
$$
\mathrm{Re}~w_1+Q(w_2)<0,
$$
where
\[
Q(w_2)=
\begin{cases}
P(1/w_2) &~\text{if}~ w_2\ne 0\\
r& \text{if}~ w_2=0.
\end{cases}
\]

Notice that $\psi\circ f$ is a local biholomorphism on $\Omega_P\cap U$. Since $\Omega_P\cap U$ is pseudoconvex near $(0,0)$, $\psi\circ f(\Omega_P\cap U)$ is pseudoconvex near $(-r,0)$. Moreover, the domain 
$$
\Omega_{Q}=\{(w_1,w_2)\in \mathbb C^2\colon \mathrm{Re}~w_1+Q(w_2)<0\}
$$
has the strong $\tilde u$-property at $(w_1^0,0)$. Therefore, it follows from Corollary \ref{Co2} that the local biholomorphisms $\psi\circ f$ and $(\psi\circ f)^{-1}$ can be extended to be $\mathcal{C}^\infty$-smooth up to the boundaries in neighborhoods of $(it_1,0)$ and $(w_1^0,0)$, respectively. However, $b\Omega_P$ and $b\Omega_Q$ are not isomorphic as CR maniflod germs at $(0,0)$ and $(-r,0)$ respectively. This is a contradiction. 

Granted the fact that $\lim_{k\to \infty}f(x_k+it_1,0)=w^0:=(w_1^0,w_2^0)\in b\Omega_P$, it follows from Lemma \ref{le3} that $\Omega_P$ is pseudoconvex near $w^0$. Moreover, again Corollary \ref{Co2} ensures that $f$ and $f^{-1}$ extend to be locally $\mathcal{C}^\infty$-smooth up to the boundaries. Hence, $\tau_{w^0}(b\Omega_P)=\tau_{(it_1,0)}(b \Omega_P)=+\infty$, which means that $w^0=(it_2,0)$ for some $t_2\in \mathbb R$ by virtue of Remark \ref{remark}. The lemma is proved.
\end{proof}

\section{Automorphism group of $\Omega_P$}
In this section, we are going to prove Theorem \ref{Th2}. To do this, let $P$ be as in Theorem \ref{Th2}. Let $p(r)$ be a $\mathcal{C}^\infty$-smooth function on $(0,\epsilon_0)~(\epsilon_0>0)$ such that the function
\[
P(z)=
\begin{cases}
e^{p(|z|)}&~\text{if}~z\in \Delta^*_{\epsilon_0} \\
0     &~\text{if}~z=0.
\end{cases}
\]
\begin{remark} \label{remark1} Since $\nu_0(P)=+\infty$, $\lim_{r\to 0^+} p(r)=-\infty$. Moreover, we observe that $\limsup_{r\to 0^+} |r p'(r)|=+\infty$, for otherwise one gets $|p(r)|\lesssim |\log(r)|$ for every $0<r<\epsilon_0$, and thus $P$ does not vanish to infinite order at $0$. In addtion, it follows from \cite[Corollary $1$]{Ninh1} that the function $P(r)p'(r)$ also vanishes to infinite order at $r=0$.
\end{remark}

 In proving Theorem \ref{Th2}, we need the following lemmas.
 \begin{lemma}[See Lemma 2 in \cite{Ninh1}]\label{lemma2} Suppose that there are $0<\alpha\leq 1$ and $\beta>0$ such that
$$
 \lim_{z\to 0}\frac{P(\alpha z)}{P(z)}=\beta.
$$ 
Then $\alpha=\beta=1$.
\end{lemma}
\begin{lemma}[See Lemma 3 in \cite{Ninh1}]\label{lemma3} Let $\beta\in \mathcal{C}^\infty(\Delta_{\epsilon_0})$ with $\beta(0)=0$. Then  
$$
P(z+z\beta(z))-P(z)=P(z)\Big(|z|p'(|z|)\big(\mathrm{Re}(\beta(z)+o(\beta(z))\big)\Big)+o((\beta(z))^2)
$$
for any $z\in  \Delta^*_{\epsilon_0}$ satisfying $z+z\beta(z)\in \Delta_{\epsilon_0}$.
\end{lemma}
In what follows, denote by $\mathcal{H}:=\{z_1\in \mathbb C\colon \mathrm{Re}(z_1)<0\}$ the left half-plane. 
\begin{lemma}\label{lemma7}
If $f\in \mathrm{Aut}(\Omega_P\cap U)\cap \mathcal{C}^\infty(\overline{\Omega_P}\cap U)$ satisfying $f_1(z_1,z_2)= a_{01} z_1+\tilde a_0(z_1)$ and $f_2(z_1,z_2)=b_{10}z_2+z_2\tilde b_1(z_1)$, where $a_{01},b_{10}\in \mathbb C^*$ with $b_{10}>0$ and $\tilde a_0,\tilde b_1\in \mathrm{Hol}(\mathcal{H}\cap U_1)\cap \mathcal{C}^\infty(\overline{\mathcal{H}}\cap U_1)$ with $\nu_0(\tilde a_0)\geq 2$ and $\nu_0(\tilde b_1)\geq 1$, where $U$ and $U_1$ are neighborhoods of the origins in $\mathbb C^2$ and $\mathbb C_{z_1}$, respectively, then $a_{01}=b_{10}=1$.
\end{lemma}
\begin{proof}
Since $f(b \Omega_P\cap U)\subset b\Omega_P$, we have
\begin{equation}\label{qtt1}
\begin{split}
\mathrm{Re}\Big(a_{01}\big(it-P(z_2)\big)+\tilde a_0\big(it-P(z_2)\big)\Big)+P\Big(b_{10}z_2+z_2 \tilde b_{1}\big(it-P(z_2)\big)\Big)\equiv 0
\end{split}
\end{equation}
on $\Delta_{\epsilon_0}\times (-\delta_0,\delta_0)$ for some $\epsilon_0,\delta_0>0$. It follows from (\ref{qtt1}) with $z_2=0$ that 
$$
\mathrm{Re}(a_{01} it)+o(t)= 0
$$
for every $t\in\mathbb R$ small enough. This yields that $\mathrm{Im}(a_{01})=0$.

On the other hand, letting $t=0$ in (\ref{qtt1}) one has 
\begin{equation}\label{qtt2}
\begin{split}
P\Big(b_{10}z_2+z_2O(P(z_2)) \Big)-\mathrm{Re}(a_{01})P(z_2)+o(P(z_2))\equiv 0
\end{split}
\end{equation}
on $\Delta_{\epsilon_0}$. This implies that $\lim_{z_2\to 0} P\big(b_{10} z_2+z_2O(P(z_2))\big)/P(z_2)=\mathrm{Re}(a_{01})=a_{01}> 0$. 

By assumption, we can write $P(z_2)=e^{p(|z_2|)}$ for all $z_2\in \Delta^*_{\epsilon_0}$ for some function $p\in\mathcal{C}^\infty(0,\epsilon_0) $ with $\lim_{t\to 0^+}p(t)=-\infty$ such that $P$ vanishes to infinite order at $z_2=0$. Therefore, by Lemma \ref{lemma3} and the fact that $P(z_2)p'(|z_2|)$ vanishes to infinite order at $z_2=0$ (cf. Remark \ref{remark1}), one gets that
 $$
  \lim_{z_2\to 0}\frac{P(b_{10} z_2)}{P(z_2)}=\lim_{z_2\to 0} \frac{P\big(b_{10} z_2+z_2O(P(z_2))\big)}{P(z_2)}=a_{01}> 0.
 $$
Hence, Lemma \ref{lemma2} ensures that $a_{01}=b_{10}=1$, which ends the proof.
\end{proof}

\begin{lemma}\label{lemma6}
If $f\in \mathrm{Aut}\big(\Omega_P\cap U\big)\cap \mathcal{C}^\infty(\overline{\Omega_P}\cap U)$ satisfying $f_1(z_1,z_2) =z_1+\tilde a_0(z_1)$ and $f_2(z_1,z_2)=z_2+z_2 \tilde b_1(z_1)$, where $\tilde a_0\in \mathrm{Hol}(U_1) $ and $\tilde b_1\in \mathrm{Hol}(\mathcal{H}\cap U_1)\cap \mathcal{C}^\infty(\overline{\mathcal{H}}\cap U_1)$ with $\nu_0(\tilde a_0)\geq 2$ and $\nu_0(\tilde b_1)\geq 1$, where $U$ and $U_1$ are neighborhoods of the origins in $\mathbb C^2$ and $\mathbb C_{z_1}$, respectively, then $f=id$.
\end{lemma}
\begin{proof}
Expand $\tilde a_0$ into the Taylor at $0$ we have
$$
\tilde a_0(z_1)=\sum_{k=2}^\infty a_{0k}z_1^k,
$$
where $a_{0k}\in \mathbb C$ for every $k\geq 2$.

Since $f$ preserves $b \Omega_P\cap U$, it follows that
\begin{equation}\label{qqt2}
\begin{split}
\mathrm{Re}\Big(\big(it-P(z_2)\big)+\sum_{k=2}^\infty a_{0k}\big(it-P(z_2)\big)^k \Big)+P\Big(z_2+\tilde b_1\big(it-P(z_2)\big)\Big)\equiv 0,
\end{split}
\end{equation}
or equivalently,
\begin{equation}\label{qqt3}
\begin{split}
P\Big(z_2+z_2\tilde b_1\big(it-P(z_2)\big)\Big)-P(z_2)+ \mathrm{Re}\Big(\sum_{k=2}^\infty a_{0k}\big(it-P(z_2)\big)^k \Big)\equiv 0
\end{split}
\end{equation}
on $\Delta_{\epsilon_0}\times (-\delta_0,\delta_0)$ for some $\epsilon_0,\delta_0>0$.

If $f_1(z_1,z_2)\equiv z_1$, then let $k_1=+\infty$. In the contrary case, let $k_1\geq 2$ be the smallest integer $k$ such that $a_{0k}\ne 0$. Similarly, if $\tilde b_1(z_1)$ vanishes to infinite order at $z_1=0$, then denote by $k_2=+\infty$. Otherwise, let $k_2\geq 1$ be the smallest integer $k$ such that $b_{1k}:=\frac{\partial^k}{\partial z_1^k}\tilde b_1(0)\ne 0$.

Notice that we may choose $t=\alpha P(z_2)$ in (\ref{qqt3})
(with $\alpha \in \mathbb R$ to be chosen later). Then one gets
\begin{equation}\label{qqt4}
\begin{split}
&P\Big(z_2+z_2 b_{1k_2}P^{k_2}(z_2)(\alpha i-1)^{k_2}+z_2o(P^{k_2}(z_2))\Big)-P(z_2)\\
&\quad +\mathrm{Re} \Big(a_{0k_1}P^{k_1}(z_2)(\alpha i-1)^{k_1}+o(P^{k_1}(z_2))\Big)\equiv 0
\end{split}
\end{equation}
on $\Delta_{\epsilon_0}\times (-\delta_0,\delta_0)$. Moreover, by Lemma \ref{lemma3} one obtains that
\begin{equation}\label{qqt4,5}
\begin{split}
& P^{k_2+1}(z_2)|z_2|p'(|z_2|)\Big(\mathrm{Re}\big(b_{1k_2}(\alpha i-1)^{k_2}+g_2(z_2)\big)\Big)\\
&\quad +P^{k_1}(z_2)\mathrm{Re} \Big(a_{0k_1}(\alpha i-1)^{k_1}+g_1(z_2)\Big)\equiv 0
\end{split}
\end{equation}
on $\Delta_{\epsilon_0}$, where $g_1,g_2\in \mathcal{C}^\infty(\Delta_{\epsilon_0})$ with $g_1(0)=g_2(0)=0$.

We remark that $\alpha$ can be chosen so that $\mathrm{Re}\big(b_{1k_2}(\alpha i-1)^{k_2}\ne 0$ and $\mathrm{Re} \big(a_{0k_1}(\alpha i-1)^{k_1}\big)\ne 0$. Furthermore, since $\limsup_{r\to 0^+} |r p'(r)|=+\infty$ (cf. Remark \ref{remark1}), (\ref{qqt4,5}) yields that $k_2+1>k_1$. However, by the fact that $P(z_2)p'(|z_2|)$ vanishes to infinite order at $z_2=0$ (see Remark \ref{remark1}) and by (\ref{qqt4,5}) one has $k_1>k_2$. Hence, we conclude that $k_1=k_2=+\infty$.

Since $k_1=k_2=+\infty$, it follows that $f_1(z_1,z_2)\equiv z_1$ and (\ref{qqt3}) is equivalent to
\begin{equation}\label{qqt5,5}
\begin{split}
P\Big(z_2+\tilde b_1\big(it-P(z_2)\big)\Big)\equiv P(z_2),
\end{split}
\end{equation}
on $\Delta_{\epsilon_0}\times (-\delta_0,\delta_0)$. Since the level sets of $P$ are circles, (\ref{qqt5,5}) implies that $\tilde b_1(z_1)\equiv 0$. Thus, the proof is complete.
\end{proof}

\begin{proof}[Proof of Theorem \ref{Th2}]
Let $f=(f_1,f_2)\in \mathrm{Aut}(\Omega_P)$. By Lemma \ref{lem6}, there exist $t_1,t_2\in \mathbb R$ such that $f$ and $f^{-1}$ extend smoothly to the boundaries near $(it_1,0)$ and $(it_2,0)$, respectively, and $f(it_1,0)=(it_2,0)$. Replacing $f$ be $T_{-t_2}\circ f \circ T_{t_1}$ we may assume that $f(0,0)=(0,0)$ and there are neighborhoods $U$ and $V$ of $(0,0)$ such that $f$ is a local CR diffeomorphism between $V\cap b \Omega_P$ and $V\cap b \Omega_P$.

For each $t\in \mathbb R$, let us define $F_t$ by setting $F_t:=f\circ R_{-t}\circ f^{-1} $. Then $\{F_t\}_{t\in \mathbb R}$ is a one-parameter subgroup of $\mathrm{Aut}(\Omega_P)\cap \mathcal{C}^\infty(\overline{\Omega_P}\cap U)$. 

By Theorem \ref{Th1}, there exists a real number $\delta$ such that $F_t=R_{\delta t}$ for all $t\in \mathbb R$. This implies that
\begin{equation}\label{tte1}
f= R_{\delta t}\circ f\circ R_{t},~\forall t\in \mathbb R .
\end{equation}
We note that if $\delta=0$, then $f=f\circ R_t$ and thus $R_t=id$ for any $t\in \mathbb R$, which is a contradiction. Hence, we can assume that $\delta\ne 0$.

We shall prove that $\delta=-1$. Indeed, by (\ref{tte1}) we have
\begin{equation}\label{tte2}
\begin{split}
 f_2(z_1,z_2)\equiv e^{i\delta t} f_2\big(z_1, z_2 e^{it}\big)
 \end{split}
\end{equation}
on a neighborhood $U$ of $(0,0)\in \mathbb C^2$ and for all $t\in \mathbb R$.

Expand $f_2$ into Taylor series, one obtains that
$$
f_2(z_1,z_2)=\sum_{n=0}^\infty b_{n} (z_1)z_2^n,
$$
where $b_{n},~n=0,2,\ldots$, are in $Hol(\mathcal{H})\cap\mathcal{C}^\infty(\overline{\mathcal{H}})$ and $b_{0}(0)=f_2(0,0)=0$. Hence, Eq. (\ref{tte2}) is equivalent to
\begin{equation}\label{tte3}
\begin{split}
\sum_{n=0}^\infty b_{n}( z_1) z_2^n\equiv \sum_{n=0}^\infty b_{n}(z_1) z_2^n e^{i\delta t+int}
\end{split}
\end{equation}
on $U$ for all $t\in \mathbb R$. This implies immediately that $b_0(z_1)\equiv 0$. Since $f$ is biholomorphism, $b_1(z_1)\not \equiv 0$. Therefore, (\ref{tte3}) yields that $\delta=-1$ and $b_{n}=0$ for every $n\in \mathbb N\setminus \{1\} $. It means that $f_2(z_1,z_2)\equiv z_2b_1(z_1)$.

We conclude that $F_t=R_{-t}$ for all $t\in \mathbb R$. This implies that
\begin{equation}\label{e1}
f= R_{-t}\circ f\circ R_{t},~\forall t\in \mathbb R,
\end{equation}
which implies that
\begin{equation*}
f_1(z_1,z_2)\equiv f_1(z_1,z_2 e^{it})
\end{equation*}
on a neighborhood $U$ of $(0,0)$ in $\mathbb C^2$ for all $t\in \mathbb R$. This yields that $f_1(z_1,z_2)=a_0(z_1)$. 

Since $f$ preserves the boundary $b\Omega_P\cap U$, we have

\begin{equation}\label{eq2014-1}
\mathrm{Re}\Big(a_0(is-P(z_2))\Big)+P\Big(z_2b_1(is-P(z_2))\Big)=0
\end{equation}
for all $(z_2,s)\in \Delta_{\epsilon_0}\times (-\delta_0,+\delta_0)$. Letting $z_2=0$ in (\ref{eq2014-1}) one gets
\begin{equation}\label{eq2014-2}
\mathrm{Re}(a_0(is))=0
\end{equation}
for all $s\in (-\delta_0,+\delta_0)$. Hence, by the Schwarz reflection principle $a_0$ extends to be holomorphic in a neighborhood of the origin $z_1=0$; we shall denote the extension by $a_0$ too and the Taylor expansion of $a_0$ at $z_1=0$ is given by
$$ 
a_0(z_1)=\sum_{m=1}^\infty a_{0m} z_1^m. 
$$
Moreover, because $f\in \mathrm{Aut}(\Omega_P)$, it follows that $a_{01}\ne 0$. From (\ref{eq2014-2}), we have
$$
\mathrm{Im}(a_{01})=0.
$$

Next, we are going to show that $b_1(0)\ne 0$. Indeed, suppose otherwise that $\nu_0(b_1)\geq 1$. Then it follows from (\ref{eq2014-1}) with $s=0$ that
$$
\lim_{z_2\to 0}\frac{P(z_2b_1(P(z_2)))}{P(z_2)}=\mathrm{Re}(a_{01})=a_{01}>0,
$$
which is impossible since 
$$
\lim_{z_2\to 0}\frac{P(z_2b_1(P(z_2)))}{P(z_2)}=\lim_{z_2\to 0}\frac{P(z_2b_1(P(z_2)))}{z_2b_1(P(z_2)}\lim_{z_2\to 0}\frac{z_2b_1(P(z_2)}{P(z_2)}=0.
$$
Hence, we conclude that
$$
f_2(z_1,z_2)=b_{10} z_2 + z_2 \tilde b_1(z_1),
$$
where $b_{10}\in \mathbb C^*$ and $\tilde b_1\in Hol(\mathcal{H})\cap\mathcal{C}^\infty(\overline{\mathcal{H}})$ with $\tilde b_1(0)=0$. In addition, replacing $f$ by $f\circ R_\theta$ for some $\theta\in \mathbb R$, we can assume that $b_{10}$ is a positive real number. 

We now apply Lemma \ref{lemma7} to obtain that $a_{01}=b_{10}=1$. Furthermore, by  Lemma \ref{lemma6} we conclude that $f=id$. Hence, the proof is complete.
\end{proof}
\section{Appendix}
We recall the following lemma, which is a version of the Hopf lemma.
\begin{lemma}[See Lemma 2.3 in \cite{Suk}] \label{hopf}
Let $\Omega\subset \mathbb C^n$ be a bounded domain with $\mathcal{C}^2$ boundary. Let $K\Subset \Omega$ be a compact set nonempty interior, and choose $L>0$. Then there exists $C=C(K,L)>0$ such that for any negative plurisubharmonic function $u$ in $\Omega$ satisfying the condition $u(z)<-L$ on $K$, the following bound holds:
$$
|u(z)|\geq C \delta_\Omega(z)~\text{for}~z\in \Omega. 
$$   
\end{lemma}
The following lemma is a slight generalization of \cite[Lemma 2.4]{Suk}.
\begin{lemma}\label{lemma12} Let $D$ and $G$ be domains in $\mathbb C^n$ with $\mathcal{C}^2$-smooth boundaries, $z_0\in bD$, $w_0\in bG$, $F_1, F_2:[0,+\infty)\to [0,+\infty)$ are nonnegative functions with $F_1(0)=F_2(0)=0$ such that $F_1$ is increasing. Assume that there is a neighborhood $U$ of $z_0$ such that for each $w\in U\cap bD$, there is a plurisubharmonic function $\psi_w$ such that
\begin{itemize}
\item[(i)] $\lim\limits_{D\times bD\ni(z,w)\to (z_0,z_0)}\psi_w(z)=0$,
\item[(ii)] $\psi_w(z)\leq -F_1(|z-w|)$,
\item[(iii)] $\psi_{\pi(z)}(z)\geq -F_2(\delta_D(z))$
\end{itemize} 
for $z\in U\cap D$. Let $f:D\to G$ be a proper map such that $w_0\in \mathcal{C}(f,z_0)$. Then there exist neighborhoods $\tilde U\subset U$ and $V$ of $z_0$ and $w_0$, respectively, such that  $\delta_G(f(z))\lesssim F_2(\delta_D(z))$ for any $z\in \tilde U\cap D$ such that $f(z)\in V\cap G$.
\end{lemma}
\begin{proof}
The proof proceeds along the same lines as in \cite[Section $2$]{Suk}, but for the reader's convenience, we shall give the detailed proof. 

For $\epsilon>0$ we consider the open set
$$
D^\epsilon=\{z\in U\cap D\colon \psi_{z_0}>-\epsilon\}.
$$ 
By virtue of (ii) there exists $\epsilon_0>0$ such that for any $\epsilon\in (0,\epsilon_0]$ one has 
$\bar D^\epsilon\subset \bar U$. Hence, the boundary $bD^\epsilon\subset  (\bar U\cap bD) \cup S^\epsilon$, where $S^\epsilon=\{z\in \bar U\cap D\colon \psi_{z_0}(z)=-\epsilon\}$.

We fix $\epsilon\in (0,\epsilon_0/2]$ and choose $\epsilon_1>0$ so that $D\cap (z_0+\epsilon_1\mathbb B)\subset D^\epsilon$, where $\mathbb B$ is the open unit ball in $\mathbb C^n$. Without loss of generality we may assume that the neighborhood $U$ is small enough such that $\delta_D(z)=|z-\pi(z)|$ for $z\in U\cap D$.  We fix a positive number $\delta$ with the properties $\epsilon_1/50<\delta<2\delta<\epsilon_1/10$ and consider the compact set $K=\bar D\cap (z_0+2\delta\bar{\mathbb B})\setminus (z_0+\delta\bar{\mathbb B})$. For $\epsilon_2<\epsilon_1/100$ we have by (ii) that 
\begin{equation*}
\begin{split}
\max\{\psi_\zeta(z)\colon z\in K,\zeta \in bD\cap (z_0+\epsilon_2 \bar{\mathbb B})\}&\leq -F_1\Big(d\big(K,bD\cap (z_0+\epsilon_2 \bar{\mathbb B}) \big)\Big)\\
  &\leq -F_1(\delta-\epsilon_2).                  
\end{split}
\end{equation*}
On the other hand, by (i) one can choose $\epsilon_2$ such that
\begin{equation*}
\begin{split}
-F_1(\delta-\epsilon_2)<\gamma:=\min\{\psi_\zeta(z)\colon z\in D\cap (z_0+\epsilon_2 \bar{\mathbb B}),\zeta \in bD\cap (z_0+\epsilon_2 \bar{\mathbb B})\}.               
\end{split}
\end{equation*}
We fix $\epsilon_2>0$. Let $\tau>0$ be such that 
$$
-F_1(\delta-\epsilon_2)<-\tau<-\tau/2<\gamma<0.
$$

We consider a smooth nondecreasing convex function $\phi(t)$ with the properties $\phi(t)=-\tau$ for $t\leq -\tau$ and $\phi(t)=t$ for $t>-\tau/2$. We set $\rho_\zeta(z)=\tau^{-1}\phi\circ \psi_\zeta(z)$. Then $\rho_\zeta(z)\mid_K=-1$ for $\zeta \in bD\cap (z_0+\epsilon_2 \bar{\mathbb B}) $, and we can extend $\rho_\zeta(z)$ to $D$ by setting $\rho_\zeta(z)=-1$ for $z\in D\setminus (z_0+2\delta \bar{\mathbb B})$. We obtain a function $\rho_\zeta(z)$, which is a negative continuous plurisubharmonic function on $D$ satisfying $\rho_\zeta(z)=-1$ on $D\setminus (z_0+\delta \mathbb B)$ and $\rho_\zeta(z)=\tau^{-1}\psi_\zeta(z)$ on $D\cap (z_0+\epsilon_2\mathbb B)$ for $\zeta \in bD\cap (z_0+\epsilon_2\mathbb B)$.

There is $\epsilon_3\in (0,\epsilon_2/2)$ such that $\pi(z)\in bD\cap (z_0+\epsilon_2\mathbb B)$ for any $z\in D\cap (z_0+\epsilon_3\bar{\mathbb B})$. We also fix a point $p\in D\cap (z_0+\epsilon_3\bar{\mathbb B})$ and define the function 
\begin{equation*}
\varphi_p(w)=
\begin{cases}
\sup\{\rho_{\pi(p)}(z)\colon z\in f^{-1}(w)\}&~\text{for}~ w\in f(D^\epsilon),\\
-1 &~\text{for}~ w\in G\setminus f(D^\epsilon).
\end{cases}
\end{equation*} 

Since $f$ is proper, the function $\varphi_p(w)$ is a continuous negative plurisubharmonic function on $G$ ( see \cite[Lemma 2.2]{Suk}).

Let $V$ be a neighborhood of the point $w_0$ such that the surface $V\cap b G$ is smooth. We fix a compact set $K\Subset f(D^{\epsilon_2})\cap V$ with nonempty interior (this is possible since $w_0\in \mathcal{C}(f,z_0)$ and $f(D^{\epsilon_2})$ is an open set). Assume that $2\max_{w\in K} \varphi_p(w)\leq -L=-L(p)$. The by Lemma \ref{hopf}, we have $|\varphi_p(w)|\geq C(L) \delta_G(w)$ for $w\in G\cap V$, where $C=C(L)>0$ depends ony on $L=L(p)$. We now show that $L$ (hence also $C$) can be chosen independent of $p$.  

We have 
\begin{equation*}
\begin{split}
\max_{w\in K}\varphi_p(w)&=\max\{\rho_{\pi(p)}(z)\colon z\in f^{-1}(w)\cap D^{\epsilon_2}, w\in K\}\\
                      &=\max\{\rho_{\pi(p)}(z)\colon z\in f^{-1}(K)\cap D^{\epsilon_2}\}.
\end{split}
\end{equation*}

Since $f$ is proper, $\mathcal{C}(f,z)\subset b G$ for $z\in U\cap b D$, and thus the set $f^{-1}(K)$ has no limit points on $U\cap b D$. Therefore, the set $K'=\overline{f^{-1}(K)\cap D^{\epsilon_2}}$ is relatively compact in $U\cap D$. If $z\in K'$, then by (iii) we have 
\begin{equation*}
\begin{split}
\max\{\psi_{\pi(p)}(z) \colon z\in K'\}&\leq -F_1\Big(\min\{|z-\zeta|\colon z\in  K',\zeta \in b D \cap (z_0+\epsilon_2\mathbb{B})\}\Big)\\
                      &=-F_1\Big(d(K',b D \cap (z_0+\epsilon_2\mathbb{B}))\Big).
\end{split}
\end{equation*}
Since the last quantity does not exceed some constant $-N<0$, we may set
$$
 2\max_{z\in K'} \rho_{\pi(p)}(z)\leq 2\tau^{-1}N:=L,
$$
and $L$ is independent of $p$.

If now $f(p)\in V\cap G$, we have by (iii) 
\begin{equation*}
\begin{split}
\delta_G(f(p))&\leq C|\varphi_p(f(p))|\leq C |\rho_{\pi(p)}(p)|\\
                      &\leq CF_2(\delta_D(p)).
\end{split}
\end{equation*}
Since $C>0$ here does not depend on $p$, we have arrived at
$$
\delta_G(f(z))\lesssim F_2(\delta_D(z))
$$
for any $z\in D\cap (z_0+\epsilon_3\mathbb B)$ such that $f(z)\in G\cap V$. This ends the proof.
\end{proof}
\section*{Acknowlegement}
 This work was completed when the author was visiting the Vietnam Institute for Advanced Study in Mathematics (VIASM). He would like to thank the VIASM for financial support and hospitality.  It is a pleasure to thank Tran Vu Khanh and Dang Anh Tuan for stimulating discussions.


\begin{thebibliography}{99}


\bibitem{Bel81} S. Bell, ``Biholomorphic mappings and the $\bar \partial$-problem", Ann. of Math. (2)  114  (1981), no. 1, 103--113.
\bibitem{BL80} S. Bell and E. Ligocka, ``A simplification and extension of Fefferman's theorem on biholomorphic mappings", Invent. Math. 57 (1980), no. 3, 283--289.
\bibitem{Ber91}
F. Berteloot,``H\"{o}lder continuity of proper holomorphic mappings", Studia Math. 100 (3) (1991), 229--235.
\bibitem{Ber92} 
F. Berteloot, ``A remark on local continuous extension of proper holomorphic mappings", The Madison Symposium on Complex Analysis (Madison, WI, 1991),  79--83, Contemp. Math., 137, Amer. Math. Soc., Providence, RI, 1992.
\bibitem{Ber94} 
F. Berteloot, ``Characterization of models in $\mathbb C^2$  by their automorphism groups", Internat. J. Math. 5 (1994), no. 5, 619--634.
\bibitem{Ber} F. Berteloot, ``Attraction de disques analytiques et continuit\'{e} Hold\'{e}rienne d'applications holomorphes propres", Topics in Compl. Anal., Banach Center Publ. (1995), 91--98.
\bibitem{By1} J. Byun, J.-C. Joo and M. Song, ``The characterization
of holomorphic vector fields vanishing at an infinite type point",
 J. Math. Anal. Appl.  387 (2012), 667--675.
\bibitem{CV14} D. Chakrabarti and K. Verma, ``Condition R and holomorphic mappings of domains with generic corners",  Illinois J. Math. 57 (2013), no. 4, 1035--1055. 
\bibitem{Ch} E. M. Chirka, ``Variations of the Hartogs theorem", (Russian) Tr. Mat. Inst. Steklova {\bf 253} (2006),  Kompleks. Anal. i Prilozh., 232--240 [ translation in  Proc. Steklov Inst. Math.  2006,  no. 2 (253), 212--220].
\bibitem{Cho92} S. Cho, ``A lower bound on the Kobayashi metric near a point of finite type in $\mathbb C^n$", J. Geom. Anal. 2 (1992), no. 4, 317--325.
\bibitem{D} J. P. D'Angelo, ``Real hypersurfaces, orders of contact, and applications",
 Ann. Math.  115 (1982), 615--637.
\bibitem{DF79} K. Diederich and J. E. Fornaess, ``Proper holomorphic maps onto pseudoconvex domains with real-analytic boundary", Ann. of Math. (2)  110  (1979), no. 3, 575--592.
\bibitem{GK91}  R. Greene and S. G. Krantz, ``Invariants of Bergman geometry and the automorphism groups of domains in $\mathbb C^n$", Geometrical and algebraical aspects in several complex variables (Cetraro, 1989),  107–136, Sem. Conf., 8, EditEl, Rende, 1991.
\bibitem{GK} R. Greene and S. Krantz, ``Techniques for studying
automorphisms of weakly pseudoconvex domains",  Math. Notes,
Vol 38, Princeton Univ. Press, Princeton, NJ, 1993, 389--410.
\bibitem{HN} A. Hayashimoto and V. T. Ninh, ``Infinitesimal CR automorphisms and stability groups of infinite type models in $\mathbb C^2$", arXiv:1409.3293, to appear in Kyoto Jourmal of Mathematics.
\bibitem{Kr01} S. Krantz,  Function theory of several complex variables. Reprint of the 1992 edition. AMS Chelsea Publishing, Providence, RI, 2001.
\bibitem{HL84}
G. M. Henkin and J. Leiterer, {\it Theory of functions on complex manifolds}, Monographs in Mathematics, Vol. 79, Birkh\"{a}user Verlag, Basel, 1984.
\bibitem{Hu95} X. Huang, ``A boundary rigidity problem for holomorphic mappings on some weakly pseudoconvex domains", Canad. J. Math. 47 (1995), no. 2, 405--420.

\bibitem{IK} A. Isaev and S. G. Krantz, ``Domains with non-compact automorphism group:
A survey",  Adv. Math. 146 (1999), 1--38.
\bibitem{IK06}  A. V. Isaev and N. G. Kruzhilin,``Proper holomorphic maps between Reinhardt domains in $\mathbb C^2$", Michigan Math. J.  54  (2006),  no. 1, 33--63.

\bibitem{Ka94} H. Kang, ``Holomorphic automorphisms of certain class of domains of infinite type", Tohoku Math. J. (2) 46 (1994), no. 3, 435--442.
\bibitem{Kr12} S. Krantz, `` The automorphism group of a domain with an exponentially flat boundary point", J. Math. Anal. Appl. 385 (2012), no. 2, 823--827.
\bibitem{Kim-Ninh} K.-T. Kim and V. T. Ninh, ``On the tangential holomorphic vector fields vanishing at an infinite type point", Trans. Amer. Math. Soc. 367 (2) (2015), 867--885.
\bibitem{Ninh1} V. T. Ninh, ``On the CR automorphism group of a certain hypersurface of infinite type in $\mathbb C^2$", Complex Var. Elliptic Equ. DOI 10.1080/17476933.2014.986656.
\bibitem{Sib81} N. Sibony, ``A class of hyperbolic manifolds", Recent developments in several complex variables (Proc. Conf., Princeton Univ., Princeton, N. J., 1979),  357--372, Ann. of Math. Stud., 100, Princeton Univ. Press, Princeton, N.J., 1981.
\bibitem{Suk} A. B. Sukhov, ``On the boundary regularity of holomorphic mappings", (Russian)  Mat. Sb.  185  (1994),  no. 12, 131--142;  translation in  Russian Acad. Sci. Sb. Math.  83  (1995),  no. 2, 541--551.
\bibitem{Khanh13} V. K. Tran, ``Boundary behavior of the Kobayashi metric near a point of infinite type", J. Geom. Anal. DOI 10.1007/s12220-015-9565-y. 


\end{thebibliography}
\end{document}